\documentclass[a4paper]{article}
\usepackage{amsmath}
\usepackage{graphicx}
\usepackage[colorinlistoftodos]
{todonotes}
\usepackage{setspace}
\usepackage{titlesec}
\usepackage{enumerate}
\usepackage{amsmath}
\usepackage{amsthm}
\usepackage{fancyvrb}
\usepackage{float}
\usepackage{lmodern}
\usepackage{footmisc}
\usepackage[utf8]{inputenc}
\usepackage[T1]{fontenc}
\usepackage{hyperref}
\hypersetup{bookmarksnumbered=true}
\usepackage{tabulary}
\usepackage{amsfonts}
\usepackage{graphicx}
\usepackage[linesnumbered,ruled,vlined]{algorithm2e}
\usepackage[ngerman,english]{babel}
\usepackage{tikz}
\usepackage{color}
\usepackage{url}
\usepackage{mathrsfs}
\usepackage[mathscr]{euscript}
\usepackage{empheq}
\usepackage{pbox}
\usepackage{csquotes}
\usepackage{afterpage}
\usetikzlibrary{matrix,backgrounds}
\pgfdeclarelayer{myback}
\pgfsetlayers{myback,background,main}

\tikzset{mycolor/.style = {line width=1bp,color=#1}}%
\tikzset{myfillcolor/.style = {draw,fill=#1}}%

\newtheorem{thrm}{Theorem}
\newtheorem{lmm}{Lemma}

\newtheorem{conj}{Conjecture}
\newtheorem{prop}{Proposition}

\newtheorem{cor}{Corollary}
\newtheorem{quest}{Question}{\bfseries}{\rmfamily}
\newtheorem{definition}{Definition}
\newcommand{\floor}[1]{\lfloor #1 \rfloor}
\newcommand{\gen}[1]{\ensuremath{\langle #1\rangle}}

\usepackage{authblk}
\title{ Characterizing 3-sets  in Union-Closed Families}
\author[1,2]{Jonad Pulaj \thanks{The work for this article has been (partly) conducted within the Research Campus MODAL funded by the German Federal Ministry of Education and Research (BMBF grant number 05M14ZAM).}}
\affil[1]{Department of Computer Science, Mathematics and Physics, Faculty of Science and Technology, The University of the West Indies, Cave Hill, St. Michael, Barbados}
\affil[2]{Dept. of Mathematical Optimization, Zuse Institute Berlin (ZIB)
Takustr. 7, 14195 Berlin, Germany}
\date{}                     
\affil[ ]{\tt jonad.pulaj@cavehill.uwi.edu}
\begin{document}
\maketitle
\begin{abstract}
A family of sets is union-closed (UC) if the union of any two sets in the family is also in the family.  Frankl's UC sets conjecture states that for any  nonempty UC family $\mathcal{F} \subseteq 2^{[n]}$ such that $\mathcal{F} \neq \left\{\emptyset\right\}$, there exists an element $i \in [n]$ that is contained in at least half the sets of $\mathcal{F}$. The 3-sets conjecture of Morris states that the smallest number of distinct 3-sets (whose union is an $n$-set) that ensure Frankl's conjecture is satisfied for any UC family that contains them
is $ \lfloor{n/2\rfloor} + 1$ for all $n \geq 4$. For an UC family $\mathcal{A} \subseteq 2^{[n]}$, Poonen's Theorem characterizes the existence of weights on $[n]$ which ensure all UC families that contain $\mathcal{A}$ satisfy Frankl's conjecture, however the determination of such weights for specific $\mathcal{A}$ is nontrivial even for small $n$. We classify families of 3-sets on $n \leq 9$ using a polyhedral interpretation of Poonen's Theorem and exact rational integer programming. This yields a proof of the 3-sets conjecture.
\end{abstract}
\section{Introduction}

Frankl's conjecture, also known as the union-closed (UC) sets conjecture, needs little introduction in combinatorial circles. Let $\mathcal{F}$ be a finite family of finite sets and denote by $U(\mathcal{F})$ the union of all sets in $\mathcal{F}$. $\mathcal{F}$ is UC if and only if for every $A,B \in \mathcal{F}$ it follows that $A \cup B \in \mathcal{F}$. Let $[n]:=\left\{1,2,\ldots,n\right\}$ and let $2^{[n]}$ denote the power set of $[n]$. Frankl's conjecture states that for any nonempty UC family $\mathcal{F}$ such that $\mathcal{F} \neq \left\{\emptyset\right\}$, there exists $i \in U(\mathcal{F})$ that belongs to at least half the sets of $\mathcal{F}$. Interest in the problem has increased following Gowers' dedicated polymath blog~\cite{Gowers}, an ongoing project with no solution in sight. 

A $k$-set is a set of cardinality $k$. Let $\mathcal{S}$ be a family of distinct 3-sets, and assume w.l.o.g that $U(\mathcal{S}) = [n]$. The 3-sets conjecture of Morris~\cite{Morris} states that the smallest cardinality of $\mathcal{S}$ such that Frankl's conjecture is satisfied for any UC family $\mathcal{F}\supset \mathcal{S}$ is $\lfloor{n/2\rfloor} + 1$, for all $n \geq 4$. 


The presence of 3-sets in UC families illustrates a well-known approach to Frankl's conjecture, namely the characterization of subfamilies of sets which ensure the conjecture holds for all families that contain them. In particular, any such subfamily ensures Frankl's conjecture is satisfied in an element that belongs to the union of sets of the subfamily. 

Following Vaughan \cite{Vaughan1}, we say that an UC family of sets $\mathcal{A} \subseteq 2^{[n]}$  is \emph{Frankl-Complete} (FC), if and only if for every UC family $\mathcal{F} \supseteq \mathcal{A}$ there exists $i \in [n]$ that is contained in at least half the sets of $\mathcal{F}$. An UC family $\mathcal{A} \subseteq 2^{[n]}$ is \emph{Non\textendash Frankl-Complete} (Non\textendash FC), if and only if there exists an UC family $\mathcal{F} \supseteq \mathcal{A}$ such that each $i \in [n]$ is in less than half the sets of $\mathcal{F}$.  It was noted early on that any UC family which contains a 1-set or a 2-set is an FC-family, as seen for example in Sarvate and Renaud~\cite{sarvate1989union}. Furthermore Sarvate and Renaud~\cite{sarvate1990improved} were the first to exhibit an UC family with 27 sets, which contained a 3-set such that all 3 of its elements were in exactly 13 sets. In this regard, this construction marks the beginning of research efforts on the characterization of FC-families with 3-sets.


Poonen~\cite{Poonen} gave a general characterization of
FC and Non\textendash FC-families, and used it for a short proof that an UC-family with a single 3-set is Non\textendash FC.
For an UC family $\mathcal{A} \subseteq 2^{[n]}$, Poonen's Theorem characterizes the existence of weights on $[n]$ which ensure all UC families that contain $\mathcal{A}$ satisfy Frankl's conjecture.  
However, using Poonen's Theorem to directly characterize specific FC-families is surprisingly difficult, as the theorem gives a constructive proof in the form of a polytope with a potentially exponential number of constraints.

Using stronger conditions than Poonen's Theorem in order facilitate the needed combinatorial analysis,
Vaughan~\cite{Vaughan3} pushed further the investigations on 3-sets whose union is an $n$-set,
and showed that any UC family with more than $\frac{2n}{3}
$ 3-sets is an FC-family. Furthermore
Vaughan~\cite{Vaughan4} announced in a conference meeting an incomplete proof (with more work underway) that any UC-family with more than $\frac{n}{2}$ 3-sets is an FC-family, but unfortunately the finished result never materialized in print, and Vaughan passed away a few years after the announcement.

Given a family of sets $\mathcal{S}$, we say that $\mathcal{S}$ \emph{generates} (or is a \emph{generator} of) $\mathcal{F}$, denoted by $\gen{\mathcal{S}}:=\mathcal{F}$, if and only if $\mathcal{F}$ is an UC family that contains $\mathcal{S}$, and there exists no UC family $\widetilde{\mathcal{F}} \subset \mathcal{F}$ such that $\mathcal{S} \subseteq \widetilde{\mathcal{F}}$. Morris \cite{Morris} introduced the following notion in his work on FC-families. Let $FC(k,n)$ denote the smallest $m$ such that any $m$ of the $k$-sets in $\left\{1,2, \ldots,n\right\}$ generate an FC-family. Morris \cite{Morris} showed that $FC(3,5)=3$, $FC(3,6)=4$, and $FC(3,7)\leq 6$. 

Finally, Mari\'c, \v Zivkovi\'c, and Vu\v ckovi\'c \cite{Serbs} formalized a combinatorial search in the interactive theorem prover Isabelle/HOL and showed that all families containing four 3-subsets of a 7-set are
FC-families. Although not explicitly mentioned in their paper, their result implies that $FC(3,7)=4$ by the lower bound on the number of 3-sets of Morris \cite{Morris}.

The characterization of specific FC-families is difficult, since previously known techniques (including computer-assisted methods, as seen in Morris~\cite{Morris}, Vaughan~\cite{Vaughan3}, and Mari\'c, \v Zivkovi\'c, and Vu\v ckovi\'c \cite{Serbs}) to overcome the difficulties of Poonen's Theorem have been exhausted.

Recently, Pulaj~\cite{PulajThesis,2017arXiv170205947P} developed a cutting plane algorithm that is able to exactly compute FC and Non-FC families using exact integer programming, allowing for the characterization of FC-families of unprecedented size, thus giving new impetus to this line of research. The cutting plane method, or Algorithm~\ref{row generation} in Section~\ref{polyhedral}, is a polyhedral reinterpretation of Poonan's Theorem which when coupled (as done in this work) with exact rational integer programming~\cite{Exact} and certificates of correctness for the branch and bound trees via VIPR~\cite{VIPR} can \emph{safely} verify all previous relevant work on FC-families, and enable the discovery of previously unknown FC-families.

In this work we prove the 3-sets conjecture of Morris~\cite{Morris} by building on previous results and proof-techniques, and closing gaps when necessary with the exact characterization of FC-families with 3-sets via Algorithm~\ref{row generation}. Before we examine the 3-sets conjecture in more detail, in Section~\ref{polyhedral} we review a polyhedral approach to Poonen's Theorem which characterizes if a given UC family $\mathcal{A}$ is FC or non--FC and yields\footnote{Our implementation\label{fnlabel} is freely available at \url{https://github.com/JoniPulaj/cutting-planes-UC-families}} Algorithm~\ref{row generation}. For complete details we refer the reader to Pulaj~\cite{PulajThesis,2017arXiv170205947P}.

\section{A Polyhedral Approach to Poonen's Theorem}\label{polyhedral}
In this section we formally introduce Poonen's Theorem and show a natural connection with polyhedral theory.
 We need the following definitions. For two families of sets $ \mathcal{A}$ and $ \mathcal{B}$, let $ \mathcal{A} \uplus \mathcal{B} := \left\{A \cup B \ | \ A \in \mathcal{A}, B \in \mathcal{B} \right\}$. For $i \in U(\mathcal{F})$ define $\mathcal{F}_i := \left\{ F \in \mathcal{F} \ | \ i \in F \right\}$. To simplify notation  we assume w.l.o.g. that $U(\mathcal{A})=[n]$. 
\begin{thrm}[Poonen 1992] \label{Poonen}
Let $\mathcal{A}$ be an UC family such that $\emptyset \in \mathcal{A}$. The following statements are equivalent:
\vspace{3mm}
 \begin{enumerate}
   \item For every UC family $ \mathcal{F} \supseteq \mathcal{A}$, there exists $i \in [n]$ such that $  |\mathcal{F}_i| \geq |\mathcal{F}|/2$.
\\
 \item There exist nonnegative real numbers $c_1, \ldots , c_n$ with $\sum_{i \in [n]}c_i =1$ such that  for  every UC family $ \mathcal{B} \subseteq 2^{[n]}$ with $\mathcal{B} \uplus  \mathcal{A} = \mathcal{B}$, the following inequality holds

\begin{equation}
\sum_{i\in [n]}c_i|\mathcal{B}_i| \geq |\mathcal{B}|/2 .  \label{poon}
\end{equation}
  \end{enumerate}
\end{thrm}

A simple observation is to view the theorem through polyhedral theory. Indeed, for a fixed UC family $\mathcal{A}$ such that $\emptyset \in \mathcal{A}$, the second statement in Theorem~\ref{Poonen} becomes a polyhedron defined as the following:
\[  P^{\mathcal{A}}:=\left\{ y \in \mathbb{R}^n \ \vline \ \begin{array}{ll}
         \sum_{i\in [n]}y_i = 1;\\
         \\
         \sum_{i\in [n]}y_i|\mathcal{B}_i| \geq |\mathcal{B}|/2 & \mbox{ $\forall \text{ UC } \mathcal{B} \subseteq 2^{[n]} : \mathcal{B} \uplus \mathcal{A} = \mathcal{B}$};\\
         \\
          y_i \geq 0 & \ \mbox{$\forall i \in [n]$};\end{array} \right\} \] 
Thus Poonen's Theorem says that a given family of sets $\mathcal{A}$ is FC if and only $P^{\mathcal{A}}$ is nonempty. Since $P^{\mathcal{A}}$ can have an exponential number of constraints, we design a cutting plane method to overcome this difficulty.

Fix an UC family $\mathcal{A}$ such that $\emptyset \in \mathcal{A}$. Let $c \in \mathbb{Z}^n_{\geq 0}$ such that $\sum_{i\in [n]}c_i \geq 1$. With every set $S \in 2^{[n]}$, we associate a variable $x_S$, i.e, a component of a vector $x \in \mathbb{R}^{2^n}$ indexed by $S$. Given a family of sets $\mathcal{F} \subseteq 2^{[n]}$, let $\mathcal{X}^{\mathcal{F}} \in \mathbb{R}^{2^n}$ denote the incidence vector of $\mathcal{F}$ defined (component-wise) as 
\[ \mathcal{X}^\mathcal{F}_S := \left\{ \begin{array}{ll}
         1 & \mbox{if $S \in \mathcal{F}$},\\
         0 & \mbox{if $S \not \in \mathcal{F}$}.\end{array} \right. \] Hence every family of sets $\mathcal{F} \subseteq 2^{[n]}$ corresponds to a unique zero-one vector in $\mathbb{R}^{2^n}$ and vice versa. Let
 $X(\mathcal{A},c)$ denote the set of integer vectors contained in the polyhedron defined by the following inequalities:
\begin{align}\
& x_S+x_T \leq 1 + x_{S \cup T} & \tiny{\forall S \in 2^{[n]}, \forall T \in 2^{[n]}} \label{one}\\
&\sum_{S\in 2^{[n]}}\left(\sum_{i \in S}c_i - \sum_{i \notin S}c_i \right)x_S +1 \leq 0 \label{two}\\
& x_S \leq x_{A\cup S} &  \forall S \in  2^{[n]}, \forall A \in \mathcal{A} \label{three}\\
& 0\leq x_S \leq 1 & \forall S\in 2^{[n]} \label{four}
\end{align}\\

\begin{prop}[Pulaj 2017] \label{FranklIntProp}
Let $\mathcal{A}$ be a UC family such that $\emptyset \in \mathcal{A}$, and let $c \in \mathbb{Z}^n_{\geq 0}$ such that $\sum_{i\in [n]}c_i \geq 1$. If $X(\mathcal{A},c) = \emptyset$, then $\mathcal{A}$ is an FC-family.

\end{prop}
This leads to the following cutting plane algorithm which can determine whether any arbitrary UC family $\mathcal{A}$ if FC or Non--FC.
\vspace{5mm}

\begin{algorithm}[H]\label{row generation}
 \SetKwInOut{Input}{Input}\SetKwInOut{Output}{Output}
 \Input{A UC family $\mathcal{A}$ such that $U(\mathcal{A})=[n]$ and $\emptyset \in \mathcal{A}$}
 \Output{$\mathcal{A}$ is an FC-family, or $\mathcal{A}$ is a Non\textendash FC-family}
 $H \leftarrow \left( \sum_{i\in [n]}y_i = 1, \ y_i \geq 0 \ \forall i \in [n] \right)$ \DontPrintSemicolon

 \While{$\exists \ \bar{y} \in H$ such that $\bar{y}=(\frac{a_1}{b_1}, \frac{a_2}{b_2}, \ldots, \frac{a_n}{b_n}) \in \mathbb{Q}^n_{\geq 0}$}{
  
   $g \leftarrow lcm(b_1,b_2, \ldots,b_n)$ \; \DontPrintSemicolon
   $c \leftarrow g\bar{y}$ \; \DontPrintSemicolon
  \If {$\exists \ \mathcal{X}^{\mathcal{B}} \in X(\mathcal{A},c)$}{
  $H \leftarrow H \cap \left(\sum_{i\in [n]}y_i|\mathcal{B}_i| \geq |\mathcal{B}|/2 \right)$\

  }
  \Else{
  \Return{} $\ \mathcal{A}$ is an FC-family
  }
}

\Return{}$ \ \mathcal{A}$ is a Non\textendash FC-family\
 \caption{Cutting planes for FC-families}
\end{algorithm}
\vspace{3mm}
\begin{thrm}[Pulaj 2017]
Let $\mathcal{A}$ be an UC family such that $U(\mathcal{A})=[n]$ and $\emptyset \in \mathcal{A}$. Then Algorithm \ref{row generation} correctly determines if $\mathcal{A}$ is an FC-family or Non\textendash FC-family.
\end{thrm}

In the next section we use Algorithm~\ref{row generation} to answer fundamental questions regarding 3-sets in UC families. In particular, we recover and complete a proof attempt on an upper bound for 3-sets 
(which always generate an FC-family) of Vaughan~\cite{Vaughan3} and consequently prove the 3-sets conjecture of Morris~\cite{Morris}. Since we exhibit many families of 3-sets, in order to improve readability, we will not use inner brackets to denote 3-sets in a given family. For example, we denote $\left\{\left\{1,2,3\right\}, \left\{2,3,4\right\}, \left\{3,4,5\right\}\right\}$ as $\left\{123, 234, 345\right\}$. When displaying 3-sets themselves we always use the usual set notation.
\section{3-sets in Union-Closed Families}
As we saw in the introduction, FC-families generated by 3-sets are well-studied and it is natural to wonder how many distinct 3-sets always generate an FC family. We first recall the following definition of Morris~\cite{Morris}.
\begin{definition}
Let $FC(k,n)$ denote the smallest positive integer $m$ such that any $m$ of the $k$-sets in $[n]$ generate an FC-family.
\end{definition}
It is not immediately clear that $FC(k,n)$ is well-defined, but the following result of Gao and Yu \cite{Gao} proves this is always the case for sufficiently large $n$ in relation to $k$.
\begin{thrm}[Gao and Yu 1998]
For all $k \geq 1$  and $n \geq 2k -2$, the UC family $\mathcal{B} \subseteq 2^{[n]}$ generated by  all the $k$-sets in $[n]$ is an FC-family, and therefore $FC(k,n) \leq \binom{n}{k}$.
\end{thrm}
Thus, with the above in mind, our question about 3-sets becomes the following: What is the \emph{minimum} number of distinct 3-sets such that any UC family that contains them satisfies Frankl's conjecture? Vaughan~\cite{Vaughan3} proved the following result.
\begin{thrm}[Vaughan 2004]
Let $\mathcal{T}$ be a family of 3-sets such that $|U(\mathcal{T})|=n \geq 4$. Suppose that $|\mathcal{T}|\geq \frac{2n}{3} +1$. Then any UC family $\mathcal{F} \supset \mathcal{T}$ satisfies Frankl's conjecture. 
\end{thrm}
Furthermore Vaughan~\cite{Vaughan3} gave an interesting but incomplete proof attempt (in the positive) of the following, which we state as a question.
\begin{quest}[Vaughan 2004]\label{vaughanq1}
Let $\mathcal{T}$ be a family of 3-sets such that $|U(\mathcal{T})|=n\geq 4$. Suppose $|\mathcal{T}|\geq \floor{\frac{n}{2}} +1$. Does this imply that any UC family $\mathcal{F}$ such that $\mathcal{F} \supset \mathcal{T}$ satisfies Frankl's conjecture? 
\end{quest}
Vaughan \cite{Vaughan4} announced in a conference meeting that an answer in the positive was near completion but unfortunately the finished result never materialized in print and the author passed away four years after the announcement. Vaughan's original proof attempt in~\cite{Vaughan3} is based on a heuristic procedure used to identify ``candidate" FC-families (recast in polyhedral terms this becomes Question 1 in Pulaj~\cite{Pulaj}). It is conceivable that her announcement was based on a near completed answer in the positive to Question 1 in Pulaj~\cite{Pulaj} for general UC families $\mathcal{A}$, which as we already showed in~\cite{Pulaj}, does not hold. Morris~\cite{Morris} explicitly stated the 3-sets conjecture as the following.
\begin{conj}[Morris 2006]\label{3sets}
$FC(3,n) = \floor{\frac{n}{2}} + 1 $ for all $n \geq 4$.
\end{conj}
Morris~\cite{Morris} proved the lower bound for the conjecture, hence a positive answer to Question~\ref{vaughanq1} implies that the 3-sets conjecture holds.
\begin{thrm}[Morris 2006]\label{lower}
$\floor{\frac{n}{2}} + 1 \leq FC(3,n)$ for all $n \geq 4$.
\end{thrm}
In what follows, we bring together nearly all known results on 3-sets in UC families. We also derive new results of interest, relying on Algorithm~\ref{row generation} when necessary. Although we limit the use of  Algorithm~\ref{row generation} in order to build on previous results, we note that all previous work on 3-sets in UC families can be directly derived and verified using Algorithm~\ref{row generation}. Our goal is to complete the proof attempt of Vaughan \cite{Vaughan3} with appropriate modifications for Algorithm~\ref{row generation} and other results we derive here.
\begin{definition}
Two families of sets contained in $2^{[n]}$ are \emph{isomorphic}, if and only if there exists a permutation of $[n]$ that transforms one into the other.
\end{definition}
Using our definition of isomorphic families of sets, since UC families have a unique minimal generator, we seek to classify UC families generated by 3-sets according to (a representative of) the isomorphism class of their generators.
\begin{definition}\label{collection}
For each $n\geq 4$, denote by $NFC(3,n)$ the largest integer $k$ such that there exists a family $\mathcal{A} \subset 2^{[n]}$  of $k$ 3-sets such that $U(\mathcal{A})=[n]$ and $\gen{\mathcal{A}}$ is a Non\textendash FC-family, and for any family $\mathcal{B} \subset 2^{[n]}$ of $k+1$ 3-sets such that $U(\mathcal{B})=[n]$, $\gen{\mathcal{B}}$ is an FC-family. Denote the collection of all such Non\textendash FC-families $\gen{\mathcal{A}}$ of cardinality $k$ by $\mathcal{T}(3,n)$.
\end{definition}
Theorem \ref{lower} implies that $NFC(3,n)$ is defined and $FC(3,n)-1 = NFC(3,n)$ for each $n\geq 4$. Thus it suffices to characterize $FC(3,n)$ to arrive at $NFC(3,n)$. 
Our goal is the classification of $\mathcal{T}(3,n)$ for all $n \leq 9$. Such a classification ensures w.l.o.g. that certain ``patterns'' are unavoidable for larger $n$. This enables the induction argument in Theorem~\ref{bigvaughan}, which leads to an upper bound on $NFC(3,n)$ and therefore an upper bound for $FC(3,n)$ for general $n$. First, we state the following results. 
\begin{thrm}[Vaughan 2003]\label{lesscommon}
Any UC family that contains a family of sets isomorphic to $\left\{135, 236, 456 \right\}$ satisfies Frankl's conjecture.
\end{thrm}
\begin{thrm}[Vaughan 2004]\label{common}
Any UC family that contains three 3-sets with a common element satisfies Frankl's conjecture.
\end{thrm}
\begin{thrm}[Poonen 1992]\label{FC(3,4)}
$FC(3,4)=3$.
\end{thrm}
\begin{cor}\label{nfc(3,4)}
$NFC(3,4)=2$.
\end{cor}
\begin{proof}
By Definition~\ref{collection} and Theorem~\ref{FC(3,4)}.
\end{proof}
Listing representatives of isomorphism classes for ``small'' families of sets is possible with the use of any computer algebra system. Furthermore, the output may be verified by hand as we outline in the appendix. In the following tables, in the left column, we will list representatives from all possible isomorphism classes for generators $\mathcal{S}$ with $NFC(3,n)$ 3-sets such that $U(\mathcal{S})=[n]$, for all $4\leq n\leq 9$. The classification of the UC families derived from the enumerated generators is achieved via Algorithm~\ref{row generation}. 

For generators which yield Non\textendash FC-families we exhibit the UC families which yield the coefficients and the right hand side scalar for an infeasible system of constraints from the second condition of Poonen's Theorem. Otherwise, in the right column, we give a reason why the generators yield an FC-family. If $X(\gen{\mathcal{S}},c)$ is empty for some vector $c \in \mathbb{Z}^{n}_{\geq 0}$ such that $\sum_{i \in [n]}c_i \geq 1$, then by Proposition~\ref{FranklIntProp} the UC family $\gen{\mathcal{S}}$ is an FC-family. In this case we display the entries of vector $c$ in the following way. The notation $i \mapsto k$ denotes $c_i = k$ for each $i \in [n]$, and some integer $k\geq 0$. This rather cumbersome notation safeguards against possible errors when reading the entries. We note that our isomorphism classes agree with those generated by Vaughan~\cite{Vaughan3}.\\\\
\begin{table}[h!]\label{table:nfc(3,4)}
\begin{center}
\begin{tabular}{|l|l|}
\hline
\multicolumn{2}{|c|}{Nonisomorphic generators $\mathcal{S}$ with two 3-sets such that $U(\mathcal{S})=[4]$}\\\hline
\hline
$\color{red}{123, 124}$ & \text{Only possible family (under permutations of $[4]$) of two 3-sets} \\\hline
\end{tabular}
\caption{Classification of 3-sets based on Corollary~\ref{nfc(3,4)}.}
\label{table:nfc(3,4)}
\end{center}
\end{table}

\begin{thrm}[Morris 2006]\label{FC(3,5)}
$FC(3,5)=3$.
\end{thrm}
\begin{cor}\label{nfc(3,5)}
$NFC(3,5)=2$.
\end{cor}
\begin{proof}
By Definition~\ref{collection} and Theorem~\ref{FC(3,5)}.
\end{proof}
\begin{table}[h!]
\begin{center}
\begin{tabular}{|l|l|}
\hline
\multicolumn{2}{|c|}{Nonisomorphic generators $\mathcal{S}$ with two 3-sets such that $U(\mathcal{S})=[5]$}\\\hline
\hline
$\color{red}{123, 145}$ & \text{Only possible family (under permutations of $[5]$) of two 3-sets} \\\hline
\end{tabular}
\caption{Classification of 3-sets based on Corollary~\ref{nfc(3,5)}.}
\label{table:nfc(3,5)}
\end{center}
\end{table}
\begin{thrm}[Morris 2006]\label{FC(3,6)}
$FC(3,6)=4$.
\end{thrm}
\begin{cor}\label{nfc(3,6)}
$NFC(3,6)=3$.
\end{cor}
\begin{proof}
By Definition~\ref{collection} and Theorem~\ref{FC(3,6)}.
\end{proof}
\begin{table}[h!]
\begin{center}
\scalebox{0.96}{
\begin{tabular}{|l|l|}
\hline
\multicolumn{2}{|c|}{Nonisomorphic generators $\mathcal{S}$ with three 3-sets such that $U(\mathcal{S})=[6]$}\\\hline
\hline
$126, 356, 456$ & \text{FC-family by Theorem \ref{common}}\\\hline
$\color{red}{123, 124, 356}$ & $\tt{infeasible}$ $\tt{ system}$ $\mathcal{P}([n] \setminus \left\{j\right\})\uplus \left\{\emptyset, 123, 124, 356\right\}, \forall j \in [n]$ \\\hline
$135, 236, 456$ & \text{FC-family by Theorem \ref{lesscommon}}\\\hline
\end{tabular}
}
\caption{Classification of 3-sets based on Corollary~\ref{nfc(3,6)}.}
\label{table:nfc(3,6)}
\end{center}
\end{table}
\begin{thrm}[Mari\'c et. al.  2012] \label{serbs}
Any UC family which contains a family $\mathcal{S}$ of four 3-sets such that $|U(\mathcal{S})|=7$ satisfies Frankl's conjecture.
\end{thrm}
\begin{cor}\label{fc(3,7)}
$FC(3,7)=4$.
\end{cor}
\begin{proof}
We arrive at $FC(3,7)=4$ by combining Theorem \ref{serbs} with Theorem \ref{lower}. 
\end{proof}
\begin{cor}\label{nfc(3,7)}
$NFC(3,7)=3$.
\end{cor}
\begin{proof}
By Definition~\ref{collection} and Theorem~\ref{fc(3,7)}.
\end{proof}
\begin{table}[h!]\label{table:nfc(3,7)}
\begin{center}
\scalebox{0.97}{
\begin{tabular}{|l|l|}
\hline
\multicolumn{2}{|c|}{Nonisomorphic generators $\mathcal{S}$ with three 3-sets such that $U(\mathcal{S})=[7]$}\\\hline
\hline
$\color{red}{123, 124, 567}$ &  $\tt{infeasible}$ $\tt{ system}$ $\mathcal{P}([n] \setminus \left\{j\right\})\uplus \left\{\emptyset, 123, 124, 567\right\}, \forall j \in [n]$ \\\hline
$127, 347, 567$ &  \text{FC-family by Theorem \ref{common}}\\\hline
$\color{red}{126, 347, 567}$ & $\tt{infeasible}$ $\tt{ system}$ $\mathcal{P}([n] \setminus \left\{j\right\})\uplus \left\{\emptyset, 126, 347, 567\right\}, \forall j \in [n]$ \\\hline
\end{tabular}
}
\caption{Classification of 3-sets based on Corollary~\ref{nfc(3,7)}.}
\end{center}
\end{table}

Using Corollary~\ref{fc(3,7)} we can also characterize $FC(3,8)$ as in the following proposition.
\begin{prop}\label{FC(3,8)}
$FC(3,8)=5$.
\end{prop}
\begin{proof}
Given a family $\mathcal{S}$ of five 3-sets such that $U(\mathcal{S})=[8]$, there is always an element $i^* \in U(\mathcal{S})$ that $i^*$ is in exactly one of the five 3-sets. Let  $A \in \mathcal{S}$ such that $i^* \in A$. Consider $\mathcal{S} \setminus \left\{A\right\}$. Then $5 \leq |U(\mathcal{S} \setminus \left\{A\right\})| \leq 7$. Assume w.l.o.g that $U(\mathcal{S} \setminus \left\{A\right\})=[n]$, for each $5 \leq n \leq 7$. Corollary \ref{fc(3,7)} with Theorem~\ref{FC(3,6)} and Theorem~\ref{FC(3,5)} yield the result. 
\end{proof}
\begin{cor}\label{nfc(3,8)}
$NFC(3,8)=4$.
\end{cor}
\begin{proof}
By Definition~\ref{collection} and Proposition~\ref{FC(3,8)}.
\end{proof}

\begin{table}[h!]
\begin{center}
\scalebox{0.88}{
\begin{tabular}{|l|l|}
\hline
\multicolumn{2}{|c|}{Nonisomorphic generators $\mathcal{S}$ with four 3-sets such that $U(\mathcal{S})=[8]$}\\\hline
\hline
$123, 478, 578, 678$ & \text{generates FC-family by Theorem \ref{common}}\\\hline
$123, 468, 578, 678$ & \text{generates FC-family by Theorem \ref{common}}\\\hline
$128, 348, 578, 678$ & \text{generates FC-family by Theorem \ref{common}}  \\\hline
$127, 348, 578, 678$ & \text{generates FC-family by Theorem \ref{common}}  \\\hline
$126, 348, 578, 678$ & \text{generates FC-family by Theorem \ref{common}}  \\\hline
$124, 348, 578, 678$ & \text{generates FC-family by Theorem \ref{common}}  \\\hline
$126, 346, 578, 678$ & \text{generates FC-family by Theorem \ref{common}}  \\\hline
$125, 346, 578, 678$ & \scriptsize{ $ 1  \mapsto 1$, $ 2  \mapsto 1$, $ 3  \mapsto 1$, $ 4  \mapsto 1$, $ 5  \mapsto 2$, $ 6  \mapsto 2$, $ 7  \mapsto 2$, $ 8  \mapsto 2$} \\\hline
$135, 237, 458, 678$ & \scriptsize{ $ 1  \mapsto 1$, $ 2  \mapsto 1$, $ 3  \mapsto 2$, $ 4  \mapsto 1$, $ 5  \mapsto 2$, $ 6  \mapsto 1$, $ 7  \mapsto 2$, $ 8  \mapsto 2$} \\\hline
$\color{red}{123, 124, 356, 678}$ & $\tt{infeasible}$ $\tt{ system}$ $\mathcal{P}([n] \setminus \left\{j\right\})\uplus \left\{\emptyset, 123, 124, 356, 678\right\}, \forall j \in [n]$ \\\hline
$\color{red}{123, 124, 567, 568}$ & $\tt{infeasible}$ $\tt{ system}$ $\mathcal{P}([n] \setminus \left\{j\right\})\uplus \left\{\emptyset, 123, 124, 567, 568 \right\}, \forall j \in [n]$ \\\hline
$123, 456, 578, 678$ & \text{since} $|U(\left\{456, 578, 678\right\})|=5$ \text{then FC-family by Theorem \ref{FC(3,5)}} \\\hline
$126, 357, 458, 678$ & $\left\{357, 458, 678\right\}$ \text{generates FC family by Theorem \ref{lesscommon}} \\\hline
\end{tabular}
}
\caption{Classification of 3-sets based on Corollary~\ref{nfc(3,8)}.}
\label{table:nfc(3,8)}
\end{center}
\end{table}

\begin{prop}\label{reduce}
Let $\mathcal{D}$ be the collection of all families of 3-sets $\mathcal{S}$ such that $|\mathcal{S}|=4$, $U(\mathcal{S})=[8]$ and $\gen{\mathcal{S}}$ is a Non\textendash FC-family. Suppose that, for each $\mathcal{S} \in \mathcal{D}$ and each 2-set $A \in \mathcal{P}([8])$, it follows that $\gen{\mathcal{S} \cup \left\{A \cup \left\{9\right\}\right\}}$ is an FC-family. Then $FC(3,9)=5$.
\end{prop}
\begin{proof}
Let $\mathcal{S}'$ be a family of five 3-sets such that $U(\mathcal{S}')=[9]$. Then there is always an element $i^* \in U(\mathcal{S}')$ such that $i^*$ is in exactly one 3-set. Let $A \in \mathcal{S}'$ such that $i^*\in A$. Then $6 \leq |U(\mathcal{S}' \setminus \left\{A\right\})|\leq 8$. Suppose $6\leq|U(\mathcal{S}' \setminus \left\{A\right\})|\leq 7$. Assume, w.l.o.g. that $U(\mathcal{S}' \setminus \left\{A\right\})=[n]$ for each $6\leq n \leq 7$. Since $|\mathcal{S}' \setminus \left\{A\right\}|=4$, Theorem~\ref{FC(3,6)} and Corollary~\ref{fc(3,7)} imply that $\gen{\mathcal{S}' \setminus \left\{A\right\}}$ is an FC-family. Hence, consider $|U(\mathcal{S}' \setminus \left\{A\right\})|= 8$ and assume, w.l.o.g. that $U(\mathcal{S}' \setminus \left\{A\right\})=[8]$. It suffices to consider $\mathcal{S}'$ such that $\gen{\mathcal{S}' \setminus \left\{A\right\}}$ is a Non\textendash FC-family. Let $\mathcal{D}$ be the collection of all families of 3-sets $\mathcal{S}$ such that $|\mathcal{S}|=4$, $U(\mathcal{S})=[8]$ and $\gen{\mathcal{S}}$ is a Non\textendash FC-family. Suppose that, for each $\mathcal{S} \in \mathcal{D}$ and each 2-set $A \in \mathcal{P}([8])$, it follows that $\gen{\mathcal{S} \cup \left\{A \cup \left\{9\right\}\right\}}$ is an FC-family. Then $FC(3,9)=5$. 
\end{proof}
Consider all nonisomorphic generators $\mathcal{S}$ with four 3-sets such that $U(\mathcal{S})=[8]$ and  $\gen{\mathcal{S}}$ is a Non\textendash FC-family. They are the following families of 3-sets (red entries in Table~\ref{table:nfc(3,8)}): $\mathcal{G}:=\left\{123, 124, 356, 678\right\} \subset \mathcal{P}([8])$, and $\mathcal{H}:=\left\{123, 124, 567, 568\right\} \subset \mathcal{P}([8])$.
\begin{cor}\label{FC(3,9)}
 $FC(3,9)=5$.
\end{cor}
\begin{proof}
Let $\mathcal{D}:=\left\{\mathcal{G}, \mathcal{H} \right\}$. In the appendix, for each $\mathcal{S} \in \mathcal{D}$ and each $A \subset \mathcal{P}([8])$ such that $|A| = 2$, we show that $\gen{\mathcal{S} \cup \left\{A \cup \left\{9\right\}\right\}}$ is an FC-family by considering all nonisomorphic $\mathcal{S} \cup \left\{A \cup \left\{9\right\}\right\}$. The result follows from Proposition~\ref{reduce}. 
\end{proof}
\begin{cor}\label{nfc(3,9)}
$NFC(3,9)=4$.
\end{cor}
\begin{proof}
By Definition~\ref{collection} and Corollary~\ref{FC(3,9)}.
\end{proof}
\begin{table}[h!]
\begin{center}
\scalebox{0.88}{
\begin{tabular}{|l|l|}
\hline
\multicolumn{2}{|c|}{Nonisomorphic generators $\mathcal{S}$ with four 3-sets such that $U(\mathcal{S})=[9]$}\\\hline
\hline
$123, 459, 689, 789$ &  \text{generates FC family by Theorem \ref{common}} \\\hline
$129, 349, 569, 789$ & \text{generates FC family by Theorem \ref{common}}\\\hline
$128, 349, 569, 789$ & \text{generates FC family by Theorem \ref{common}} \\\hline
$\color{red}{123, 457, 689, 789}$ & $\tt{infeasible}$ $\tt{ system}$ $\mathcal{P}([n] \setminus \left\{j\right\})\uplus \left\{\emptyset, 123, 457, 689, 789\right\}, \forall j \in [n]$ \\\hline
$\color{red}{123, 124, 356, 789}$ & $\tt{infeasible}$ $\tt{ system}$ $\mathcal{P}([n] \setminus \left\{j\right\})\uplus \left\{\emptyset, 123, 124, 356, 789\right\}, \forall j \in [n]$ \\\hline
$\color{red}{125, 345, 689, 789}$ & $\tt{infeasible}$ $\tt{ system}$ $\mathcal{P}([n] \setminus \left\{j\right\})\uplus \left\{\emptyset, 125, 345, 689, 789\right\}, \forall j \in [n]$ \\\hline
$123, 468, 569, 789$ & $\left\{468, 569, 789\right\}$ \text{generates FC family by Theorem \ref{lesscommon}}\\\hline
$127, 348, 569, 789$ & \scriptsize{ $ 1  \mapsto 1$, $ 2  \mapsto 1$, $ 3  \mapsto 1$, $ 4  \mapsto 1$, $ 5  \mapsto 1$, $ 6  \mapsto 1$, $ 7  \mapsto 2$, $ 8  \mapsto 2$, $ 9  \mapsto 2$} \\\hline
\end{tabular}
}
\caption{Classification of 3-sets based on Corollary~\ref{nfc(3,9)}.}
\label{table:nfc(3,9)}
\end{center}
\end{table}
In the rest of this section we closely follow and complete the proof attempt of Vaughan~\cite{Vaughan3} by recasting it in the proposed framework of this chapter. Thus, we are able to close the ``gaps'' with our classification of $\mathcal{T}(3,n)$ for all $4 \leq n \leq 9$ and Algorithm~\ref{row generation} where necessary.
\begin{prop} \label{Vaugh}
Let $\mathcal{S}:= \mathcal{G} \cup \left\{\left\{a,b,c\right\}\right\}$, such that $\left\{a,b,c\right\}$ is any 3-set for distinct $a,b,c \in \mathbb{N}_{1}$. Suppose $\left\{a,b,c\right\} \cap [8] \neq \emptyset$. Then $\gen{\mathcal{S}}$ is an FC-family.
\end{prop}
\begin{proof}
Suppose $\left\{a,b,c\right\} \cap [8] \neq \emptyset$ and recall Theorem \ref{common}. Since a family with more than two 3-sets which share an element is FC, we note that $|\mathcal{S}_1|, |\mathcal{S}_2|, |\mathcal{S}_3|,|\mathcal{S}_6|\geq 2$. Therefore if $\left\{a,b,c\right\} \cap \left\{1,2,3,6\right\} \neq \emptyset$, it follows that $\gen{\mathcal{S}}$ is an FC-family. Hence it suffices to consider the cases when $\left\{a,b,c\right\} \cap \left\{4,5,7,8\right\} \neq \emptyset$ in order to determine whether $\left\{a,b,c\right\} \cap [8] \neq \emptyset$ implies that $\gen{\mathcal{S}}$ is an FC-family.

Suppose w.l.o.g. that $a=4$ and consider $\mathcal{D}:=\left\{123,124,356, 4bc \right\}$. Suppose $6 \leq |U(\mathcal{D})|\leq 7$ and assume w.l.o.g. that $U(\mathcal{D})=[n]$ for each $6\leq n\leq 7$. From Theorem~\ref{FC(3,6)} and Corollary~\ref{fc(3,7)} we see that $FC(3,n)= 4$ for each $6\leq n\leq 7$. Hence it follows that $\gen{\mathcal{D}}$ is an FC-family, and therefore $\mathcal{S} \supset \mathcal{D}$ generates an FC-family. Suppose that $|U(\mathcal{D})|$ is maximal. Therefore, assume w.l.o.g $U(\mathcal{D})=[8]$, $b=7$ and $c=8$. Then $\mathcal{D}$ is isomorphic to $\left\{125, 346, 578, 678\right\}$\footnote{In cycle notation, we permute the ground set of $\mathcal{D}$ in the following way, $(18)(27)(36)(45)$, to arrive at (not in the same order) $\left\{125, 346, 578, 678\right\}$.}, and from Table~\ref{nfc(3,8)} we see that $\gen{\mathcal{D}}$ is an FC-family. Therefore $\mathcal{S} \supset \mathcal{D}$ generates an FC-family. 

Suppose $8\leq |U(\mathcal{S})| \leq 9$ and assume w.l.o.g. that $U(\mathcal{S})=[n]$ for each $8\leq n\leq 9$. From Proposition~\ref{FC(3,8)} and Corollary~\ref{FC(3,9)} we arrive at $FC(3,n)= 5$ for each $8\leq n\leq 9$. Therefore, it suffices to check whether the following values of $a,b,c$ lead to FC-families: 5,9,10 and 7,9,10 and 8,9,10. We note that for $a,b,c$ equal to 7,9,10 and 8,9,10, we arrive at isomorphic families of sets. Therefore we consider the following two cases:
\begin{itemize}
\item Suppose $\mathcal{S}=\left\{123,124,356,678, 59(10) \right\}$. Let $\mathcal{S}':= \left\{123, 356, 678, 59(10) \right\}$. Then since $|U(\mathcal{S}')|=9$, we assume w.l.o.g. that $U(\mathcal{S}')=[9]$ and arrive at $\left\{123, 345, 567, 489 \right\}$ which is isomorphic to $\left\{127, 348, 569, 789\right\}$\footnote{To arrive from  $\left\{127, 348, 569, 789\right\}$ to (not in the same order) $\left\{123, 345, 567, 489 \right\}$ we permute the ground set in cycle notation: $(1)(2)(3957)(48)(6)$.}. From Table~\ref{nfc(3,9)} we see that $\gen{\mathcal{S}'}$ is an FC-family, and therefore $\mathcal{S} \supset \mathcal{S}'$ generates an FC-family.
\item Suppose $\mathcal{S}=\left\{123,124,356,678, 79(10) \right\}$. Let $c \in \mathbb{Z}^{10}_{\geq 0}$ such that \\$c=(6,6,8,4,5,7,5,4,2,2)$. Then $X(\gen{\mathcal{S}},c)$ is empty, hence by Proposition~\ref{FranklIntProp} the UC family $\gen{\mathcal{S}}$ is an FC-family.

\end{itemize}

\end{proof}
 Proposition \ref{Vaugh} says that if we add any 3-set $\left\{a,b,c\right\}$ to $\mathcal{G}$ such that $\left\{a,b,c\right\} \cap [8] \neq \emptyset$, we arrive at an FC-family. Hence the next corollary follows immediately.
\begin{cor} \label{8coll}
Let $\mathcal{S}$ be a family of 3-sets. Suppose $\gen{\mathcal{S} \cup \mathcal{G}}$ is a Non\textendash FC-family. Then, for each $ S \in \mathcal{S}$ it follows that $S \cap [8] = \emptyset$.
\end{cor}
Consider all nonisomorphic generators $\mathcal{S}$ with three 3-sets such that $U(\mathcal{S})=[6]$ and  $\gen{\mathcal{S}}$ is a Non\textendash FC-family. Let $\mathcal{I}:=\left\{123, 124, 356\right\} \subset \mathcal{P}([6])$. From the red entry in Table~\ref{table:nfc(3,6)}, we see that $\mathcal{I}$ is the only such family.
\begin{cor}\label{used}
Let $\mathcal{S}$ be a family of 3-sets. Define $\mathcal{T}:= \mathcal{S} \cup \mathcal{I}$. Suppose $\gen{\mathcal{T}}$ is a Non\textendash FC-family. Then $|\mathcal{T}_4| = 1$, and either $|\mathcal{T}_5| = 1$ or $|\mathcal{T}_6| = 1$.
\end{cor}
\begin{proof}
Suppose $\gen{\mathcal{T}}$ is a Non\textendash FC-family and $|\mathcal{T}_4| = 2$. Observe from the second paragraph of the proof of Proposition \ref{Vaugh} that $ \mathcal{I} \subset  \mathcal{D}$. Then it follows that $\mathcal{T}$ generates an FC-family and we arrive at a contradiction. Therefore $|\mathcal{T}_4| = 1$.

 Suppose $\gen{\mathcal{T}}$ is a Non\textendash FC-family and $|\mathcal{T}_6| = 2$. Let $\left\{6,b,c \right\}$ be a 3-set for distinct $b,c \in \mathbb{N}_{1}$. Suppose $\mathcal{T}=\mathcal{I} \cup \left\{\left\{6,b,c\right\}\right\}$. Suppose that $|\mathcal{T}_5| = 2$. Then $6\leq |U(\mathcal{T})| \leq 7$. Assume w.l.o.g. that $U(\mathcal{T})=[n]$ for each $6\leq n\leq 7$. From Theorem~\ref{FC(3,6)} and Corollary~\ref{fc(3,7)} we see that $FC(3,n)= 4$ for each $6\leq n\leq 7$. Hence it follows that $\gen{\mathcal{T}}$ is an FC-family, which is a contradiction. 

Therefore, either  $|\mathcal{T}_5| = 1$ or $|\mathcal{T}_6| = 1$.
\end{proof}
Consider all nonisomorphic generators $\mathcal{S}$ with four 3-sets such that $U(\mathcal{S})=[9]$ and  $\gen{\mathcal{S}}$ is a Non\textendash FC-family. They are the following families of 3-sets (red entries from Table~\ref{table:nfc(3,9)}): $\mathcal{J}:=\left\{123, 457, 689, 789\right\} \subset \mathcal{P}([9])$, $\mathcal{K}:=\left\{123, 124, 356, 789\right\} \subset \mathcal{P}([9])$, and $\mathcal{L}:=\left\{125, 345, 689, 789\right\} \subset \mathcal{P}([9])$.
\begin{lmm} \label{6lemma}
Let $\mathcal{S}$ be a family of 3-sets. Suppose $\gen{\mathcal{S}}$ is a Non\textendash FC-family such that $\left\{1,2,3\right\} \in \mathcal{S}$, $|\mathcal{S}_1|=|\mathcal{S}_2|=|\mathcal{S}_3| = 2$.  Then $\mathcal{S}$ contains a permutation of $\mathcal{I}$.
\end{lmm}
\begin{proof}
Suppose $\left\{1, 2, 3\right\} \in \mathcal{S}$ and $\gen{\mathcal{S}}$ is a Non\textendash FC-family such that $|\mathcal{S}_1|=|\mathcal{S}_2|=|\mathcal{S}_3| = 2$. Let $\left\{1, 2, a\right\}$ be a 3-set such that $a \in \mathbb{N}_{1}$. Suppose $\left\{1, 2, a\right\}\in \mathcal{S}$. Since $|S_3| = 2$, we may assume $\mathcal{S}$ contains $\mathcal{D}:= \left\{123,12a,bc3\right\}$, for distinct $b,c \in \mathbb{N}_{1}$. Suppose $4\leq|U(\mathcal{D})|\leq 5$, and assume w.l.o.g. $U(\mathcal{D})=[n]$ for each $4\leq n \leq 5$. Then, Theorem~\ref{FC(3,4)} and Theorem~\ref{FC(3,5)} imply that $\gen{\mathcal{D}}$ is an FC-family. Therefore $\mathcal{S} \supset \mathcal{D}$ generates an FC-family, which is a contradiction. Hence $|U(\mathcal{D})|=6$, and assume w.l.o.g. that $U(\mathcal{D})=[6]$. Under permutations of the ground set, this implies that $\mathcal{D}$ is isomorphic to $\mathcal{I}$. Therefore if $\mathcal{S}$ contains any other set besides $\left\{1, 2, 3\right\}$ which contains any two of the elements 1, 2 and 3, then $\mathcal{S}$ contains some permutation of $\mathcal{I}$.

Suppose $\mathcal{S}$ does not contain another set besides $\left\{1, 2, 3\right\}$ with any two of the elements 1, 2 and 3. Then, we may assume $\mathcal{S}$ contains $\mathcal{D}:= \left\{123,145,2ab,3cd\right\}$ for distinct $a,b \in \mathbb{N}_{1}$ and distinct $c,d \in \mathbb{N}_{1}$. 

Suppose that $\left\{a,b,c,d\right\} \cap [5] \neq \emptyset$, and suppose $5\leq|U(\mathcal{D})|\leq 7$. Assume w.l.o.g. that $U(\mathcal{D})=[n]$ for each $5 \leq n \leq 7$. Then Theorem~\ref{FC(3,5)}, Theorem~\ref{FC(3,6)} and Corollary~\ref{fc(3,7)} imply that $\gen{\mathcal{D}}$ is an FC-family. Therefore $\mathcal{S} \supset \mathcal{D}$ generates an FC-family, which is a contradiction.

Suppose that $\left\{a,b,c,d\right\} \cap [5] \neq \emptyset$ and $|U(\mathcal{D})|= 8$. Assume w.l.o.g that $U(\mathcal{D})=[8]$. Then $\mathcal{D}$ is not isomorphic\footnote{This is easiest to see if we identify with each family a binary matrix where each column represents a set. It suffices to consider the square block structure in the upper left hand corner of the matrices corresponding to $\mathcal{H}$ and $\mathcal{G}$. The block structure cannot be recovered in $\mathcal{D}$ by an appropriate choice of $a,b,c,d$ without clearly implying that $\mathcal{D}$ is nonisomorphic to $\mathcal{H}$ and $\mathcal{G}$.} to either $\mathcal{H}$ or  $\mathcal{G}$, and hence $\gen{\mathcal{D}}$ is an FC-family. Therefore $\mathcal{S} \supset \mathcal{D}$ generates an FC-family, which is a contradiction. 

It follows that $\left\{a,b,c,d\right\} \cap [5] = \emptyset$ and hence $|\mathcal{D}|=9$. Assume w.l.o.g. that $U(\mathcal{D})=[9]$. Examining  $\mathcal{J}$,  $\mathcal{K}$, and  $\mathcal{L}$ we see that the only family which contains a 3-set $\left\{a,b,c\right\}$ for distinct $a,b,c \in \mathbb{N}_{1}$ such that $a,b$ and $c$ are in exactly 2 sets is $\mathcal{K}$, and $\mathcal{I}$ is contained in $\mathcal{K}$. 
\end{proof}
\begin{thrm}\label{bigvaughan}
NFC(3,n) $\leq n/2$ for all $n \geq 4$.
\end{thrm}
\begin{proof}
We proceed by induction. We have shown the statement to be true for $4 \leq n \leq 9$. We assume it's true for all positive integers up to and including $n -1$, and show that it holds for $n$. Suppose $\mathcal{S}$ is a family of 3-sets such that $\gen{\mathcal{S}}$ is a Non\textendash FC-family and $U(\mathcal{S})=[n]$. Suppose $|\mathcal{S}|=k$. Let $a$ be the number of distinct positive integers $i\in [n]$ such that $|\mathcal{S}_i|= 2$, and $b$ the number of distinct positive integers $i\in [n]$ such that $|\mathcal{S}_i| = 1$. Theorem \ref{common} implies that $|\mathcal{S}_i|$ can only equal one or two, and we arrive at $a+b=n$ and $2a +b = 3k$. It follows that $a=3k -n$ and $b = 2n - 3k$. If $b \geq k$, we arrive at $b = 2n - 3k \geq k$, which implies that $k \leq n/2$. Suppose $b < k$. This implies there are more 3-sets than distinct positive integers $i$ such that $|\mathcal{S}_i|=1$. Therefore, if we removed all the 3-sets such that $|\mathcal{S}_i|=1$, we would be left with at least one 3-set such that all its elements appear exactly twice. We can safely assume this is the set $\left\{1,2,3\right\}$, and $|\mathcal{S}_1|=|\mathcal{S}_2|=|\mathcal{S}_3|=2$. By Lemma \ref{6lemma}, we may safely assume that $\mathcal{S}$ contains $\mathcal{I}$ and by Corollary \ref{used} we assume w.l.o.g. that $|\mathcal{S}_4|= |\mathcal{S}_5|= 1$. We now consider the parity of $n$.
\begin{enumerate}
\item Suppose $n$ is even. Let $\mathcal{T}:= (\mathcal{S} \setminus \left\{\left\{1,2,3\right\}\right\}) \setminus \left\{\left\{1,2,4\right\}\right\}$. Since $|\mathcal{S}_4|= 1$ and $|\mathcal{S}_1|=|\mathcal{S}_2|=|\mathcal{S}_3|=2$, it follows that $|U(\mathcal{T})| = n -3$. Assume w.l.o.g. that $U(\mathcal{T})=[n-3]$. Since $|\mathcal{T}|=k-2$, we arrive at $k-2 \leq NFC(3,n-3)$. Using our induction hypothesis we arrive at $k-2 \leq (n-3)/2$, which implies that $k \leq (n+1)/2$ and it follows that $k \leq n/2$, since $k$ is an integer and $n$ is even.
\item Suppose $n$ is odd. We consider the following two cases:
\begin{itemize}
\item $|\mathcal{S}_6|=1$. Let $\mathcal{T}:= \mathcal{S} \setminus \left\{\left\{3,5,6\right\}\right\}$. As previously $|U(\mathcal{T})| = n-2$, $|\mathcal{T}|=k-1$. Assume w.l.o.g that $U(\mathcal{T})=[n-2]$ and by applying the induction hypothesis we arrive at $k-1 \leq (n-2)/2$, which implies that $k \leq n/2$.
\item $|\mathcal{S}_6|=2$. Since $\mathcal{S}$ contains $\mathcal{I}$ and $|\mathcal{S}_1|=|\mathcal{S}_2|=|\mathcal{S}_3|=2$, $|\mathcal{S}_4|=|\mathcal{S}_5|=1$, then we may safely assume $\mathcal{S}$ contains $\mathcal{G}$, up to isomorphism. By Corollary \ref{8coll}, we see that $|\mathcal{S}_7|=|\mathcal{S}_8|=1$. Let $\mathcal{T}:= \mathcal{S} \setminus \left\{\left\{6,7,8\right\}\right\}$, and as in the previous case we arrive at 
$k \leq n/2$.
\end{itemize}
\end{enumerate} 
\end{proof}
\begin{cor}\label{3-setsproblem}
Let $\mathcal{T}$ be a family of 3-sets such that $|U(\mathcal{T})|=n \geq 4$. Suppose $|\mathcal{T}|\geq \floor{\frac{n}{2}} +1$. Then any UC family $\mathcal{F}$ such that $\mathcal{F} \supset \mathcal{T}$ satisfies Frankl's conjecture.	
\end{cor}
\begin{proof}
Follows directly from Theorem~\ref{bigvaughan} and Definition~\ref{collection}. 
\end{proof}
\begin{cor}
$FC(3,n) = \floor{\frac{n}{2}} + 1 $ for all $n \geq 4$.
\end{cor}
\begin{proof}
Follows directly from Corollary~\ref{3-setsproblem} and Theorem~\ref{lower}. 
\end{proof}
\bibliographystyle{plain}
  \bibliography{PulajCuttingPlanes}
\appendix
\section{Appendix}
Next, we briefly outline how to verify the output of a computer algebra system for isomorphism classes of generators for FC-families. We note, that this is possible to do by hand only when the number of classes and the ground set is small. For nontrivial fixed $n,m,k$ we want to partition all families $\mathcal{S}$ of $k$-sets such that $|\mathcal{S}|=m$, $U(\mathcal{S})=[n]$ according to their isomorphism class, and then display a representative from each. Thus in order to ensure that the output of a computer algebra system is correct, it suffices to ensure that the representative families of generators are pairwise nonisomorphic and that by adding the cardinalities of each of the isomorphism classes we recover the cardinality of the collection of all families $\mathcal{S}$ as above. This can be done by standard counting techniques and is easy but  lengthy.

In the next two tables, we show that all nonisomorphic families of 3-sets constructed from families of sets isomorphic to $\mathcal{G}$ and $\mathcal{H}$ (as needed for the proof of Corollary~\ref{FC(3,9)} by adding an appropriate 3-set yield generators for FC-families. The rest of the tables give nonisomorphic minimal generators of previously unknown FC-families.
\\
\begin{table}[h!]
\begin{center}
\scalebox{0.85}{
\begin{tabular}{|l|l|}
\hline
\multicolumn{2}{|c|}{Nonisomorphic generators $\left\{124, 346, 578, 678\right\} \cup \left\{A \cup \left\{9\right\}\right\}$ for each $A \in \mathcal{P}([8])$ with $|A| = 2$}\\\hline
\hline
124, 346, 578, 678, 789 & \text{generates FC-family by Theorem \ref{common}}\\\hline
124, 346, 578, 678, 589 & \text{generates FC-family by Theorem \ref{common}}\\\hline
124, 346, 578, 678, 459 & \text{generates FC-family by Theorem \ref{common}}\\\hline
 124, 346, 578, 678, 489  & \text{generates FC-family by Theorem \ref{common}}\\\hline
 124, 346, 578, 678, 369  & \text{generates FC-family by Theorem \ref{common}}\\\hline
 124, 346, 578, 678, 349  & \text{generates FC-family by Theorem \ref{common}}\\\hline
 124, 346, 578, 678, 269  & \text{generates FC-family by Theorem \ref{common}}\\\hline
 124, 346, 578, 678, 249  & \text{generates FC-family by Theorem \ref{common}} \\\hline
 124, 346, 578, 678, 689  & \text{generates FC-family by Theorem \ref{common}}\\\hline
 124, 346, 578, 678, 569  & \text{generates FC-family by Theorem \ref{common}} \\\hline
 124, 346, 578, 678, 469  & \text{generates FC-family by Theorem \ref{common}}\\\hline
 124, 346, 578, 678, 389  & \text{generates FC-family by Theorem \ref{common}}  \\\hline
124, 346, 578, 678, 289  & \text{generates FC-family by Theorem \ref{common}}  \\\hline
 124, 346, 578, 678, 239  & \text{$ 124, 346, 239 $ generates FC-family by Theorem \ref{lesscommon}} \\\hline
 124, 346, 578, 678, 359  & \text{since} $|U(\left\{346, 578, 678, 359\right\})|=7$  \text{then FC-family by Corollary~\ref{fc(3,7)}}\\\hline
 124, 346, 578, 678, 259  & \text{$ 346, 578, 678, 259 $ is isomorphic to $ 125, 346, 578, 678 $ in Table~\ref{nfc(3,8)}} \\\hline
124, 346, 578, 678, 129  & \scriptsize{ $ 1  \mapsto 2$, $ 2  \mapsto 2$, $ 3  \mapsto 2$, $ 4  \mapsto 3$, $ 5  \mapsto 1$, $ 6  \mapsto 3$, $ 7  \mapsto 2$, $ 8  \mapsto 2$, $ 9  \mapsto 1$} \\\hline
\end{tabular}
}
\caption{Proof of Corollary~\ref{FC(3,9)} where $\left\{124, 346, 578, 678\right\}$ is isomorphic to $\mathcal{G}$}
\end{center}
\end{table}
\begin{table}[h!]
\begin{center}
\scalebox{0.85}{
\begin{tabular}{|l|l|}
\hline
\multicolumn{2}{|c|}{Nonisomorphic generators $\left\{134, 234, 578, 678\right\}\cup \left\{A \cup \left\{9\right\}\right\}$ for each $A \in \mathcal{P}([8])$ with $|A| = 2$}\\\hline
\hline
 134, 234, 578, 678, 789  &  \text{generates FC family by Theorem \ref{common}} \\\hline
 134, 234, 578, 678, 689  & \text{generates FC family by Theorem \ref{common}} \\\hline
 134, 234, 578, 678, 489  & \text{generates FC family by Theorem \ref{common}}\\\hline
 134, 234, 578, 678, 469  & \text{generates FC-family by Theorem \ref{common}}  \\\hline
 134, 234, 578, 678, 569  & \text{since}  $|U(\left\{578, 678, 569\right\})|=5$ \text{then FC-family by Theorem~\ref{FC(3,5)}}\\\hline
 134, 234, 578, 678, 269  & \scriptsize{ $ 1  \mapsto 1$, $ 2  \mapsto 3$, $ 3  \mapsto 2$, $ 4  \mapsto 2$, $ 5  \mapsto 1$, $ 6  \mapsto 3$, $ 7  \mapsto 2$, $ 8  \mapsto 2$, $ 9  \mapsto 2$} \\\hline
\end{tabular}
}
\caption{Proof of Corollary~\ref{FC(3,9)} where $\left\{134, 234, 578, 678\right\}$ is isomorphic to $\mathcal{H}$ }
\end{center}
\end{table}

\end{document}